%% file: main_cyliccubic.tex
\documentclass[12pt,a4paper,leqno]{amsart}

\usepackage{amssymb,hyperref}
\usepackage{array,float}
\usepackage{amsmath,amsthm,amsbsy,amscd,mathrsfs}
\usepackage{graphicx}
\usepackage{caption}
\usepackage{cite}

\usepackage{mathtools}

\newcommand{\Mod}[1]{\ (\mathrm{mod}\ #1)}

\DeclareMathOperator{\Pic}{Pic}
\DeclareMathOperator{\Div}{Div}
\DeclareMathOperator{\To}{T^0}

\DeclareMathOperator{\Tr}{Tr}  
\DeclareMathOperator{\co}{covol}

\makeatletter
\newcommand\suchthat{%
 \@ifstar
  {\mathrel{}\middle|\mathrel{}}
  {\mid}%
}
\makeatother

\def\bf{\mathbf}

\def\rm{\textrm}

\theoremstyle{plain}
 \newtheorem{theorem}{Theorem}[section]
 \newtheorem{proposition}{Proposition}[section]
 \newtheorem{lemma}{Lemma}[section]
 \newtheorem{corollary}{Corollary}[section]
\theoremstyle{definition}

\theoremstyle{remark}
\newtheorem{remark}{Remark}[section] 
\numberwithin{equation}{section}

 \newtheorem{definition}{Definition}[section]

\title[THE SIZE FUNCTION FOR CYCLIC CUBIC FIELDS]{THE SIZE FUNCTION FOR CYCLIC CUBIC FIELDS}

\author[Ha Thanh Nguyen Tran]{Ha Thanh Nguyen Tran} 
\address{Department of Mathematics and Statistics,
University of Calgary\\
2500 University Drive NW\\
Calgary,  Alberta, Canada T2N 1N4.}
\email{hatran1104@gmail.com}

\author[Peng Tian]{Peng Tian} 
\address{Department of Mathematics,
	East China University of Science and Technology\\
	Meilong Road 130\\
	200237, Shanghai, P. R. China.}
\email{tianpeng@ecust.edu.cn}


\begin{document}

\maketitle

\begin{abstract}
	The size function for a number field is an analogue of the dimension of the Riemann-Roch spaces of divisors on an algebraic curve. It was conjectured to attain its maximum at the the trivial class of Arakelov divisors. This conjecture was proved for many number fields with unit groups of rank one. Our research confirms that the conjecture also holds for cyclic cubic fields, which have unit groups of rank two. 
\end{abstract}


\section{Introduction}\label{sec1a}
\input{int}

\section{Preliminaries}\label{sec2}
\input{cubic}

\subsection{The ring of integers}\label{cubic}
\input{ringofintegers}

\subsection{The unit lattice}\label{logunit}\input{logunit}

\subsection{Arakelov divisors}
\input{basicara1}

\subsection{The Arakelov class group}\label{Pic}
\input{basicara2}

\subsection{The function $h^0$}\label{h0}
\input{basicara3}

\subsection{The road map of the proof of Theorem \ref{thmmain}}
\input{road_map}

\subsection{Some estimates}\label{sec1}
Let $L$ be a lattice in $\mathbb{R}^3$ with the length of the shortest vectors  $\lambda$.  
\input{estimate}

\input{sec3}

\section{Proof of Theorem \ref{thmmain} when $I$ is not principal}\label{case1}
\input{case1}

\section{Proof of Theorem \ref{thmmain} when $I$ is principal}\label{case2}
\input{case2}

\subsection{Case $p \ge 9$ and  $0.324096 < \|w\| \le \sqrt{3} \lambda_1/2$ }\label{case2a}
\input{case2a}

\subsection{Case $p \ge 9$ and $0.170856 \le \|w\| \le 0.324096$}\label{case2b}
\input{case2b}

\subsection{Case $p =7$ and $0.170856  \le \|w\| \le  \sqrt{3} \lambda_1/2$}\label{case2c}
\input{case2c}

\subsection{Case $p \ge 7$ and  $0< \|w\| < 0.170856$}\label{case2d}
\input{case2d1}

\input{case2d2}

\section{Previous and further work}\label{sec5}
\subsection{A comparison to previous work}
\input{compare}
\subsection{Further work}
\input{furtherwork}

\section*{Acknowledgement}
The author would like to thank  Ren\'{e} Schoof for discussion and very useful comments and Camilla Hollanti for her great hospitality during the time a part of this paper was written. The author also would like to thank the reviewers and Andrew Fiori for their insightful comments that helped improve the manuscript. 

The author is financially supported by the Academy of Finland (grants  $\#$276031, $\#$282938, and $\#$283262). Support from the Pacific Institute for the Mathematical Sciences (PIMS), and the European Science Foundation under the COST Action IC1104 is also gratefully acknowledged.




\end{document}

%% file: int.tex
The function $h^0$ for a number field $F$ was introduced in \cite{ref:3}, which is also called the ``size function"  for $F$ (see \cite{ref:14,ref:15,ref:27,ref:21}). This function is well defined on the Arakelov class group $\Pic^0_F$ of $F$ (see \cite{ref:4}). 
Concerning the maximality of $h^0$, the following conjecture was proposed \cite{ref:3}.

\textit{Conjecture}. 
Let $F$ be a number field that is Galois over $\mathbb{Q}$ or over an imaginary quadratic number field. Then the function $h^0$ on $\Pic^0_F$ assumes its maximum on the trivial class $O_F$ where $O_F$ is the ring of integers of $F$.

The conclusion of this conjecture holds for 
quadratic fields \cite{ref:14}, certain pure cubic fields \cite{ref:15} and quadratic extensions of complex quadratic fields \cite{Tran2}. In this paper,  we prove that this conjecture also holds for all cyclic cubic fields. We remark that, in contrast to the above-cited works, in the case we handle here the unit group has rank two, rather than rank one. Explicitly, we will prove the following theorem.

\begin{theorem}\label{thmmain}	
	Let $F$ be a cyclic cubic field. Then the function $h^0$ on $\Pic^0_F$ has its unique global maximum at the trivial class $D_0=(O_F, 1)$. 	
\end{theorem}

In general, the conclusion of this theorem is not true for cubic fields that are not Galois. For instance, it does not hold in the case of the totally real cubic field defined by the polynomial $X^3+X^2-3X-1$.
 
The assumption that $F$ is cyclic Galois is thus important. The Galois property allows us to make use of several invariance properties (see Lemmas \ref{lengf} and \ref{h0sym}) which are crucial in our proofs of Propositions \ref{equiv} and \ref{taylor1}. Moreover, this condition allows for an explicit description of the ring of integers $O_F$ (see Proposition \ref{OF}), and the unit group $O_F^\times$ (see Proposition \ref{hexan}). This allows for the efficient calculation of lower bounds on the lengths of elements of $O_F$ (when viewed as a lattice in $\mathbb{R}^3$, see Proposition \ref{minlength}).

We first introduce Arakelov divisors, the Arakelov class group, the size function $h^0$, and some properties of cyclic cubic fields in Section \ref{sec2}. The proof of Theorem \ref{thmmain} is presented in Sections \ref{case1} and \ref{case2}. Finally, we give a comparison to previous work and then discuss the futher work in Section \ref{sec5}.

%% file: cubic.tex
From now on, we fix a cyclic cubic field $F$ with $O_F$  the ring of integers and $G= \langle \sigma \rangle$ the Galois group of $F$. Let $p$ be the conductor of $F$. The discriminant of $F$ is $\Delta = p^2$. 

Denote by
$$\mathbb{R}^{\times} = \{\alpha \in \mathbb{R}: \alpha \neq 0 \} \hspace{0.5cm} \text{  and  }  \hspace{0.5cm}  \mathbb{R}^{\times}_{+} = \{\alpha \in \mathbb{R}: \alpha >0 \} .$$

The map $\Phi: F \longrightarrow \mathbb{R}^3$ is defined by 
$$\Phi(f) = (\sigma^i(f))_{0 \le i \le 2} = (f, \sigma(f), \sigma^2(f)) \text{ for all } f \in F.$$
Note that in this paper, we often identify a fractional ideal $I$ of $F$ with its image $\Phi(I)$ that is also a lattice in $\mathbb{R}^3$. Indeed, each $f \in I$ is identified with $\Phi(f) \in \mathbb{R}^3$. Thus $\|f\|^2= \|\Phi(f) \|^2= \sum_{i=0 }^2 |\sigma^i(f)|^2.$
Moreover, a lattice is called hexagonal if it is isometric to the lattice $M \cdot \mathbb{Z}[\zeta_3]$ for some $M \in \mathbb{R}^{\times}_{+}$ and a primitive cube root of unity $\zeta_3$.

\begin{remark}\label{conductor}
	The conductor $p$ of $F$ has the form
	$$p = p_1 p_2 \cdots p_r,$$
	where $r \in \mathbb{Z}_{>0}$ and $p_1, \cdots , p_r$ are distinct integers from the set
	$$ \{9\} \cup  \{q (\text{prime}) \equiv 1\Mod 3\} = \{7, 9, 13, 19, 31, 37, \cdots \}.$$
	See \cite{Hasse} for more details.
\end{remark}

Since $F$ is a cyclic extension, the following fact is easily seen. Note that this result will be used many times in the next sections.

\begin{lemma}\label{lengf}
	Let $f \in F$. Then $\|f \| = \| \sigma(f)\| = \|\sigma^2(f)\|$.
\end{lemma}

\begin{proposition}\label{hexan}
	Let $L$ be a lattice of rank two and let $\tau$ be an isometry of this lattice such that $\tau^2 +\tau +1 =0$. Then $L$ is a hexagonal lattice. 
\end{proposition}
\begin{proof}
	The lattice $L$ can be seen as a $\mathbb{Z}[\tau]/(\tau^2 + \tau + 1)$-module.
	The ring $\mathbb{Z}[\tau]/(\tau^2 + \tau + 1)$ is isomorphic to $\mathbb{Z}[\zeta_3]$ which is a PID.
	It follows that $L$ is free of rank 1.
	Now pick a generator $\omega$ of $L$. Let $M = \|\omega\|$.  The homomorphism
	$$M \cdot \mathbb{Z}[\zeta_3] \longrightarrow L$$
	given by  $M(a + b\zeta_3) \longmapsto  a \omega + b\tau(\omega)$   for $a,b\in \mathbb{Z}$,
	is an isomorphism of $\mathbb{Z}[\zeta_3]$-modules and even an isometry of lattices. Thus, this proposition is proved. 	
\end{proof}

%% file: ringofintegers.tex
The structure of $O_F$ can be described as below.

\begin{proposition}\label{OF}
There exists some $f \in O_F$ such that $\Tr(f) = f + \sigma(f) + \sigma^2(f) =0$ and one of the following holds.
\begin{enumerate}
\item[i)] $O_F = \mathbb{Z} \oplus \mathbb{Z}[\sigma] \cdot f$  or 
\item[ii)] $O_F \supset \mathbb{Z} \oplus \mathbb{Z}[\sigma] \cdot f$  and $[O_F : (\mathbb{Z} \oplus \mathbb{Z}[\sigma] \cdot f)] =3$.
\end{enumerate}

\end{proposition}

\begin{proof}
Consider the group homomorphism $\Tr:  O_F \longrightarrow \mathbb{Z}$ that takes each $g \in O_F$ to its trace $\Tr(g)$.

Denote by $K = \ker(\Tr)$ and  $R=\mathbb{Z}[\sigma]/(\Tr)$.
One can see that $K$ is a free module of rank 1 over $R$. In other words, there exits some $f \in O_F$ such that $\Tr(f) = f + \sigma(f) + \sigma^2(f) =0$ and $K=\mathbb{Z}[\sigma] \cdot f$. 
In addition, since $\sigma$ is an isometry of $K$ and $\sigma^2 +\sigma +1 = \Tr = 0$ on $K$, Proposition \ref{hexan} says that $K$ is a hexagonal lattice.

The image $\Tr(O_F)$ contains $3 = \Tr(1)$. Therefore $\Tr$ is surjective or its image has index 3.
Moreover, $K$ is a rank 2 sublattice of $O_F$  that is orthogonal to $\mathbb{Z}$.
Thus, $O_F = \mathbb{Z} \oplus K = \mathbb{Z} \oplus \mathbb{Z}[\sigma] \cdot f$ if $\Tr$ is not surjective (case i)).
In case $\Tr$ is surjective, the lattice $ \mathbb{Z} \oplus K =\mathbb{Z} \oplus \mathbb{Z}[\sigma] \cdot f$  has index 3 in $O_F$ (case ii)). 

\end{proof}

\begin{proposition}\label{minlength}
We have  $\|g\|^2 \ge \frac{2p}{3}$ for all $g \in O_F \backslash \mathbb{Z}$. 
\end{proposition}

\begin{proof}
With the notations of Proposition \ref{OF}, we set $L_1 = \mathbb{Z} \oplus K= \mathbb{Z} \oplus \mathbb{Z}[\sigma] \cdot f$. 
Since $K$ is orthogonal to $\mathbb{Z}$, $f$ is a shortest vector in $L_1\backslash \mathbb{Z}$.
The fact that $K$ is a hexagonal lattice leads to the following.
\begin{equation} \label{eq21}
 \co(L_1 )=\|1\| \|f\|\|\sigma(f)\| (\sqrt{3}/2)=3 \|f\|^2/2 .  
  \end{equation}
  
 There are two cases.
\begin{itemize}
\item[i)] If $O_F = L_1$ then 
\begin{equation} \label{eq22}
\co(L_1)= \co(O_F) = p.
\end{equation} 
It follows from \eqref{eq21} and \eqref{eq22} that $\|f\|^2= 2p/3$.
The result is then implied since $f$ is a shortest vector of $O_F\backslash\mathbb{Z}$.
\item[ii)] If $O_F \supset L_1$  and $[O_F : L_1] =3$ then 
\begin{equation} \label{eq23}
\co(L_1)= 3\co(O_F) = 3p.
\end{equation}
It follows from \eqref{eq21} and \eqref{eq23} that $\|f\|^2= 2p$.
\end{itemize}
Observe that $O_F = L_1 + \mathsf{x}$, where $\mathsf{x}$ is an element of the form
$\mathsf{x} = [a + bf + c\sigma(f)]/3$  for certain $a,b,c\in \{-1,0,1\}$ that has trace not
divisible by 3. Since  $\Tr(\mathsf{x})=a$, by replacing $\mathsf{x}$ by $\pm \mathsf{x} + l$ for some integer $l$ if necessary, one may assume that $a = 1$. Therefore $\mathsf{x} = [1 + bf + c\sigma(f)]/3$.

Since $K=\mathbb{Z} f \oplus \mathbb{Z} \sigma(f)$ is a hexagonal lattice, we have 
$$\|bf +c\sigma(f)\|^2=(b^2-bc+c^2)\|f\|^2= 2(b^2-bc+c^2)p.$$
By Pythagoras theorem,
$$\|\mathsf{x}\|^2= (\|1\|^2 + \|bf +c\sigma(f)\|^2 )/9=[3 + 2(b^2-bc+c^2)p]/9.$$
Since $\|\mathsf{x}\|^2$ is an integer, 9 must divide $3 + 2(b^2-bc+c^2)p$. 
Now, if one of $b,c$ is 0 or if $b=c$, then the expression  $b^2-bc+c^2$ is 1 and hence
$3 + (b^2-bc+c^2)\|f\|^2 = 3 + 2p$  is divisible by 9. This is impossible by Remark \ref{conductor}. 
Hence $b=-c = \pm 1$. As the result,  
 $\mathsf{x}= [1 \pm (f -\sigma(f))]/3$.
Accordingly, $\|\mathsf{x}\|^2 = (3 + 3\|f\|^2)/9= (1 + 2p)/3$. It is easy to see that this is the length squared of the shortest vectors of $O_F\backslash\mathbb{Z}$, which completes the proof.
 \end{proof}

%% file: logunit.tex
The map $\log: F^\times \longrightarrow  \mathbb{R}^3$ is defined as below.
$$\log(f):= (\log|\sigma^i(f)|)_{0 \le i \le 2} \in \mathbb{R}^3 \text{ for all } f \in F^\times.$$
We set
$$\mathcal{H} = \{ (v_0, v_1, v_2) \in \mathbb{R}^3: v_0 + v_1 + v_2 = 0 \},$$
 a plane in $\mathbb{R}^3$, and
$$\Lambda= \log(O_F^\times)= \{(\log|\sigma^i(\varepsilon)|)_{i=0}^2: \varepsilon \in O_F^\times\}.$$
Note that $\Lambda$ is a full rank lattice contained in  $\mathcal{H}$ by the Dirichlet's unit theorem. Let $\lambda_1$ be the length of the shortest vectors of $\Lambda$.

\begin{remark}\label{lambdahexan}
	Since $\sigma$ an isometry of $\Lambda$ and $\sigma^2 +\sigma +1 = \Tr = 0$ on $\Lambda$, 
	one obtains that $\Lambda $ is a hexagonal lattice by applying Proposition \ref{hexan}.
 \end{remark}
By Remark \ref{lambdahexan}, one can assume that $\Lambda $ has a $\mathbb{Z}$-basis containing two shortest vectors $b_1 = \log \varepsilon_1, b_2=\log \varepsilon_2$ for some $\varepsilon_1, \varepsilon_2 \in O_F^\times$ and with $\|b_1\| = \|b_2\|=\|b_2-b_1\|$ (Figure \ref{pic:hexagonal}). Denote by 
$$\mathcal{F} = \left\{\alpha_1 \cdot b_1 + \alpha_2 \cdot b_2: \alpha_1, \alpha_2 \in \left(-\frac{1}{2}, \frac{1}{2}\right] \right\},  \text{ and }$$
$$B(w)= \{ \mathbf{x} \in O_F^{\times}: \|\log \mathbf{x} -w\| < \lambda_1\} \text{ for each } w \in \mathcal{F}.$$

The set $B(w)$ can be described by the following lemma.
\begin{lemma}\label{Bomega}
Let $w \in \mathcal{F}$. Then $\# B(w) \le 8$. Moreover, 
$$B(w) \subset \{\pm 1, \pm \mathbf{x}_1, \pm\mathbf{x}_2, \pm \mathbf{x}_3\} \subset O_F^{\times} \text { where }$$ 
$$\|\log \mathbf{x}_1-w\| \ge  3\lambda_1/16 , \|\log\mathbf{x}_2-w\| \ge  \lambda_1/2  \text{ and } \|\log \mathbf{x}_3-w\| \ge  \sqrt{3}\lambda_1/2.$$
\end{lemma}

\begin{figure}[h]
	\centering
	\includegraphics[width=0.6\linewidth]{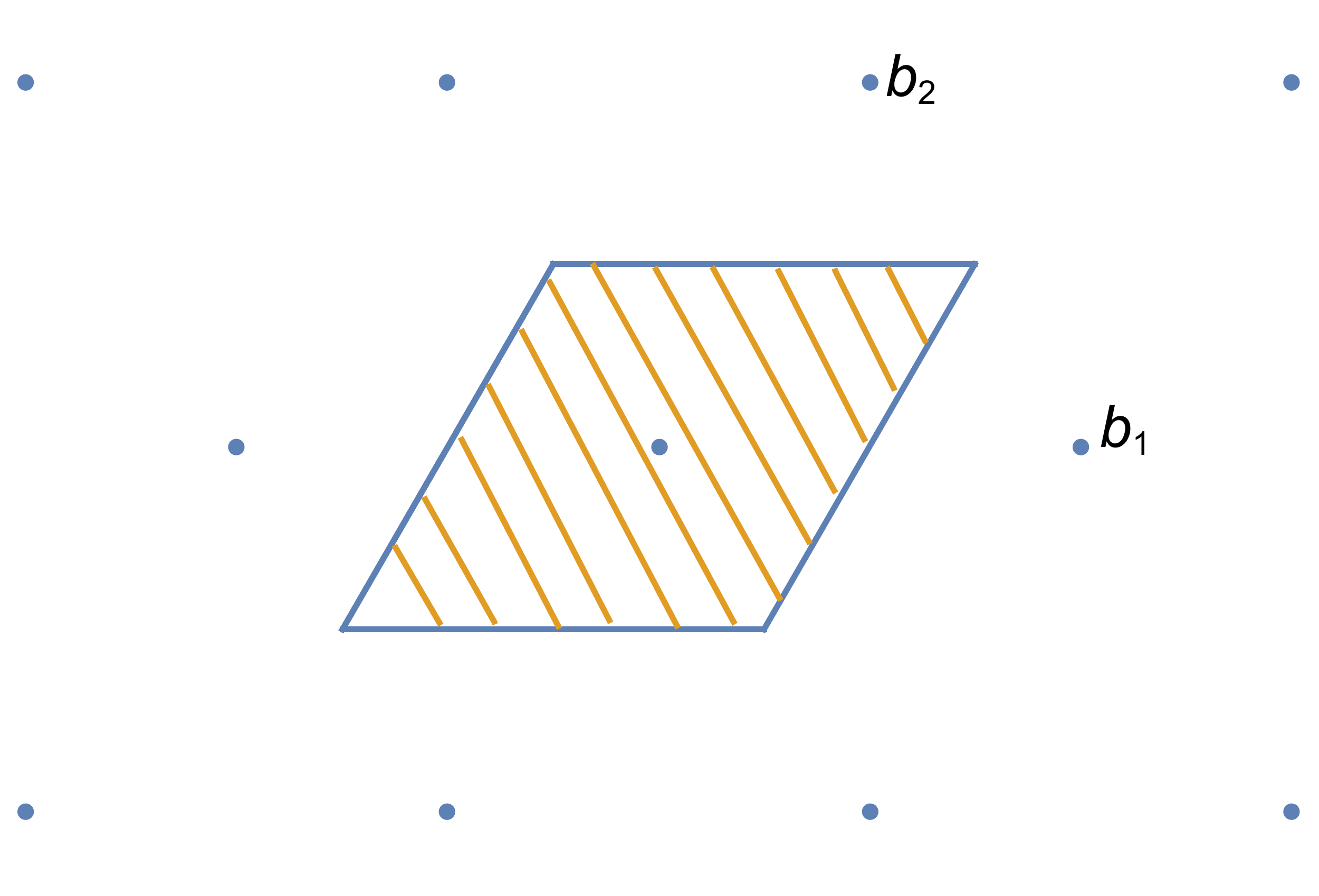}
	\caption{The lattice $\Lambda$ and $\mathcal{F}$ (the shaded area). \label{pic:hexagonal}}
\end{figure}   
 
\begin{proof}
Since $\Lambda$ is a hexagonal lattice (see Remark \ref{lambdahexan}), it has at most four points from which the distance to $w$ strictly less than $\lambda_1$ (see Figure \ref{pic:hexagonal}). Each point $v \in \Lambda$ has the form $\log(\mathbf{x})$ or $\log(-\mathbf{x})$ for some $\mathbf{x} \in O_F^{\times}$. Therefore, there are at most 8 points of $O_F^{\times}$ in $B(w)$. 

Assume that $B(w) \subset \{\pm 1, \pm \mathbf{x}_1, \pm\mathbf{x}_2, \pm \mathbf{x}_3\} \subset O_F^{\times}$ with 
$$\|\log \mathbf{x}_1-w\| \le \|\log\mathbf{x}_2-w\| \le  \|\log \mathbf{x}_3-w\|.$$
The lower bound for each $\|\log \mathbf{x}_i-w\|$ is easily observed since $\Lambda$ is  hexagonal (see Figure \ref{pic:hexagonal}).
\end{proof}

\begin{lemma}\label{lambda1}
If $p=7$ then  $\lambda_1 \ge 1.025134$. Moreover,  $\lambda_1 \ge 1.296382$ when $p \ge 9$.
\end{lemma}

\begin{proof}
If $p=7$ or $p=9$ then $F$ is a simplest cubic field for which a pair of fundamental units $\{\epsilon_1, \epsilon_2\} $ can be computed easily \cite{shanks1974simplest}. The vectors $\log \epsilon_1$ and $\log \epsilon_2$ form a $\mathbb{Z}$-basis for the lattice $\Lambda = \log O_F^{\times}$. Using this basis, one can easily find a shortest vector of $\Lambda$ and its length. Here one obtains that $\lambda_1 \ge 1.025134$ when $p=7$ and 
 $\lambda_1 \ge 1.303291$ when $p=9$.

We now consider the case in which $p \ge 13$. Let $v  \in \Lambda\backslash \{0\}$. Then $v = \log\mathbf{x}$ for some $\mathbf{x} \in O_F^{\times}\backslash\{\pm 1\}$. Proposition \ref{minlength} says that $\|\mathbf{x}\|^2 \ge 9$ since $p \ge 13$. Hence $\|v\|=\|\log \mathbf{x}\| \ge 1.296382$. This equality holds for all nonzero vectors of $\Lambda$, therefore $\lambda_1 \ge 1.296382$.
\end{proof}

%% file: basicara1.tex
Let  $u = (u_0, u_1, u_2) \in \mathbb{R}^3$. The \textit{norm} of $u$ is defined by $N(u):= u_0 u_1 u_2$. 

 \begin{definition}
 	An \textit{Arakelov divisor} of $F$ is a pair $D=(I,u)$ where $I$ is a fractional ideal of $F$ and $u$ is an arbitrary element in $(\mathbb{R}^{\times}_{+})^3 \subset \mathbb{R}^3$. 
 \end{definition}
The Arakelov divisors of $F$ form an additive group denoted by $\Div_F$. 
The \textit{degree} of $D = (I,u)$ is defined by $\deg(D): = \log{(N(u) N(I))}$. 
 Denote by $ u f:= u \cdot \Phi(f) = (u_i \cdot \sigma^i(f)) \in \mathbb{R}^3$ for all $f \in I$. In terms of coordinates, one has 
  $$\|u f\|^2 =\|u \cdot \Phi(f)\|^2 = \sum_{i=0 }^2  u_i^2 \cdot |\sigma^i(f)|^2.$$
Let $ u I: =\{u f: f \in I\}\subset \mathbb{R}^3$. Then it is a lattice with the  metric inherited from $\mathbb{R}^3$. We call $u I$ the \textit{lattice associated} to $D$.

To each element $f \in F^\times$ is attached a \textit{principal} Arakelov divisor $(f):=(f^{-1} O_F, |f|)$ where 
$f^{-1} O_F= \{f^{-1} g: g \in O_F\}$ is the principal (fractional) ideal generated by $f^{-1}$ and $|f|: =|\Phi(f)|= (|\sigma^i(f)|)_{0 \le i \le 2} \in \mathbb{R}^3$. It has degree $0$ by the product formula. See \cite{ref:4, ref:3} for full details about Arakelov divisors.

%% file: basicara2.tex
All Arakelov divisors of degree 0 form a group, denoted by $\Div^0_F$. 
Similar to the Picard group of an algebraic curve, we have the following definition.
\begin{definition}
The \textit{Arakelov class group }  $\Pic^0_F$ is the quotient of $\Div^0_F$ by its subgroup of principal divisors.
\end{definition}
 
 We define
$$\To = \mathcal{H} / \Lambda.$$
Thus $\To$ is a real torus of dimension $2$. Each class $v = (v_0, v_1, v_2) \in  \To$ can be embedded into $\Pic^0_F$ as the class of the divisor $D_v= (O_F, u)$ with $u = (e^{-v_{i}})_{i}$. Therefore, $\To$ can be viewed as a subgroup of $\Pic^0_F$. 

 Denoting by $Cl_F$ the class group of $F$, the structure of $\Pic^0_F$ can be seen by the following proposition.

\begin{proposition}\label{prop:structure1}
The map that sends the Arakelov class represented by a divisor $D = (I, u)$ to the ideal class of $I$ is a homomorphism from $ \Pic^0_F$ to the class group $Cl_F$ of $F$. 
It induces the exact sequence 
\begin{align*}
 0 \longrightarrow \To \longrightarrow \Pic^0_F \longrightarrow Cl_F \longrightarrow 0  .
\end{align*}
\end{proposition}
\begin{proof}
See Proposition 2.2 in \cite{ref:4}.
\end{proof}

The group $\To$ is the connected component of the identity of the topological group $\Pic^0_F$. 
 Each class of Arakelov divisors in $\To$ is represented by a divisor of the form $D = (O_F, u)$ for some  
$  u \in (\mathbb{R}^{\times}_{+})^3$ and $N(u)=1$.  Here $u$ is unique up to multiplication by units $\varepsilon \in O_F^\times$ (see Section 6 in \cite{ref:4}).

%% file: basicara3.tex
 Let $D=(I,u)$ be an Arakelov divisor of $F$. 
We denote by  $$k^0(D) = \sum_{f \in I}e^{-\pi\|u f\|^2} = \sum_{ \mathsf{x} \in u I}e^{-\pi\| \mathsf{x} \|^2}\hspace*{1cm}  \text{ and } \hspace*{1cm}  h^0(D)=\log(k^0(D)).$$ 
 The function $h^0$ is  well defined on $\Pic^0_F$ and analogous to the dimension of the Riemann-Roch space $H^0(D)$ of a divisor $D$ on an algebraic curve. See Proposition 4.3 in \cite{ref:4} and \cite{ref:3} for full details.
 
 \begin{lemma} \label{h0sym}
 	The function $h^0$ on $\To$ is invariant under the action of $\sigma$. In other words
 	$$h^0(D) = h^0(\sigma(D)) \text{ for all } D \in \To.$$
 \end{lemma}
 \begin{proof}
 	Let $D = (O_F, u) \in \To$ with $u=(u_0, u_1, u_2) \in (\mathbb{R}_{>0})^3$. Then $\sigma(D) = (\sigma(O_F), \sigma(u))=(O_F, (u_1, u_2, u_0))$ since $F$ is Galois.
 	Hence 
 	$$\|\sigma(u) \sigma(f)  \|^2= u_1^2 [\sigma(f)]^2  + u_2^2 [\sigma^2(f)]^2+u_0^2 [\sigma^0(f)]^2=\sum_{ i=0}^{2}u_i^2 [\sigma^i(f)]^2 = \|u f\|^2  \text{ for all } f \in O_F.$$
 	Consequently, 
 	$$k^0(\sigma(D)) = \sum_{f \in O_F}e^{-\pi\|\sigma(u) f\|^2}=\sum_{\sigma(f)  \in O_F}e^{-\pi\|\sigma(f)  \sigma(u)\|^2}=\sum_{f  \in O_F}e^{-\pi\|u f\|^2}=k^0(D).$$
 \end{proof}

%% file: road_map.tex
Let $D=(I, u)$ be an Arakelov divisor of degree $0$ with $L$ the lattice associated to $D$. We write  
$$k^0(D) = 1 + S_{1} + S_{2}  \text{ where } $$ 
$$ S_{1} = \sum_{\substack{ f \in I\backslash\{0\} \\ \| u f\|^2 < 3\cdot 2^{2/3}}}e^{-\pi \|u f\|^2 }
\hspace*{0.5cm}\text{ and } \hspace*{0.5cm} S_{2} = \sum_{\substack{ f \in I \\ \| u f\|^2 \geq 3\cdot 2^{2/3} }}e^{-\pi \|u f\|^2 }
= \sum_{\substack{ \mathsf{x} \in L\backslash\{0\} \\ \| \mathsf{x}\|^2 \ge 3\cdot 2^{2/3}}}e^{-\pi \|\mathsf{x}\|^2 }.$$
To prove Theorem \ref{thmmain}, we will show that $k^0(D)<k^0(D_0)$.
An upper bound on $S_2$ is given in Corollary \ref{sum}, and Lemma \ref{S1} provides an upper bound on $S_1$ that is sufficient to obtain Theorem \ref{thmmain}. When $I$ is not principal, then $S_1 = 0$ --- this is the case of Section \ref{case1} --- and the theorem is proved.  Otherwise, the class of $D$ in $\Pic_F^0$ is the image of an element $w + \Lambda \in \To$ with $w \in \mathcal{F}$ (Section \ref{case2}). If the length of $w$ is not too short, then Lemma \ref{Bomega2} shows that $S_1$ contains two identical collections of at most four terms each, and each term can be bounded above using Lemma \ref{Bomega}; this is done in cases \ref{case2a}--\ref{case2c} in Section \ref{case2}. Then Lemma \ref{S1} once again yields the result of Theorem \ref{thmmain}. Finally, if the length of $w$ is short (case \ref{case2d} in Section \ref{case2}), then the bound on $S_1$ in Lemma \ref{S1} cannot be attained; instead, we apply Propositions \ref{equiv} and \ref{Gat1}--\ref{sum3} to prove Theorem \ref{thmmain} directly.

The next section provides upper bounds on $S_1$ and $S_2$ which are used in the proof of Theorem \ref{thmmain}.

%% file: estimate.tex
\begin{lemma}\label{err}
	Let $M \geq \lambda^2\geq a^2>0$ and let $\alpha >0$.	
 	Then  
 		$$\sum_{\substack{ \mathsf{x}  \in L   \\ \|\mathsf{x} \|^2 \geq  M }}e^{-\alpha \|\mathsf{x} \|^2 } \leq \alpha  \int_{M}^{\infty} \!  \left( \left( \frac{2\sqrt{t}}{a} +1\right)^3 - \left( \frac{2\sqrt{M}}{a} -1\right)^3 \right) e^{- \alpha  t}\, \mathrm{d}t.$$
\end{lemma}

\begin{proof}
This proof is obtained by using an argument similar to the proof of Lemma 3.2 in \cite{Tran2} with the degree of the number field $n =3$ and by replacing $\pi$ with $\alpha$. 
\end{proof}

\begin{corollary}\label{sum} 
Assume  that $\lambda^2  \geq 3$. We have  
$$\sum_{\substack{ \mathsf{x} \in L   \\ \|\mathsf{x} \|^2 \geq 3\cdot 2^{2/3} }}e^{-\pi \|\mathsf{x} \|^2 } < 137.648\cdot 10^{-6},$$
$$\sum_{\substack{ \mathsf{x} \in L   \\ \|\mathsf{x} \|^2 \geq 10 }}e^{-(\pi-1/2) \|\mathsf{x}  \|^2 } < 0.001 \cdot 10^{-6} \hspace*{1cm} \text{ and } \sum_{\substack{ \mathsf{x} \in L   \\ \|\mathsf{x} \|^2 \geq 10 }}e^{-1.568075 \|\mathsf{x}  \|^2 } < 23.399 \cdot 10^{-6}.$$
\end{corollary}
\begin{proof}
	Use Lemma \ref{err} with $a=\sqrt{3}$.
\end{proof}

%% file: sec3.tex
	
The following lemma is applied to prove Theorem \ref{thmmain} in cases \ref{case1} and \ref{case2a}--\ref{case2c}.

\begin{lemma}\label{S1}
If $S_1  < 0.000147634$, then $k^0(D_0)> k^0(D)$.
\end{lemma}

\begin{proof}
 Let $f \in I\backslash\{0\} $. Then $|N(f)|\ge N(I)$. Moreover 
 $N(I) N(u) =1$ since $\deg(D) = 0$. Hence
		$$\|u f\|^2 \geq 3 |N(u f)|^{2/3} = 3 |N(u) N( f)|^{2/3} = 3 \left(\frac{|N(f)|}{N(I)}\right)^{2/3} \geq 3.$$
		This holds for any nonzero $f \in I$.
Therefore, the length of the shortest vectors of the lattice $L$ is $ \lambda \geq \sqrt{3}$. 
Corollary \ref{sum} says that 
	$S_{2} < 137.648\cdot 10^{-6}.$

Subsequently, one obtains
$$k^0(D) =1+ S_1 + S_2 <  1+S_1+ 137.648\cdot 10^{-6}.$$
In addition, it is obvious that  $$k^0(D_0) > 1 + 2 e^{-3 \pi}.$$
 Thus, to prove that $k^0(D_0)> k^0(D)$, it is sufficient to prove the following.
 \begin{equation*}
 S_1  < 2 e^{-3 \pi} -137.648\cdot 10^{-6} \le 0.000147634.
 \end{equation*} 
\end{proof}

%% file: case1.tex
	Since $I$ is not principal, $|N(f)|/N(I)\geq 2$ for all  $f \in  I \backslash \{0\}$. 
	 In addition, $N(I) N(u) =1$ since $\deg(D) = 0$. Therefore
	$$\|u f\|^2 \geq 3 |N(u f)|^{2/3} = 3 |N(u) N( f)|^{2/3} = 3 \left(\frac{|N(f)|}{N(I)}\right)^{2/3} \geq 3\cdot 2^{2/3}.$$
	This inequality holds for any nonzero $f \in I$. Hence  $S_1 = 0$ and Theorem \ref{thmmain} is proved by Lemma \ref{S1}.

%% file: case2.tex
In this section we will prove Theorem \ref{thmmain} when the ideal $I$ is principal. We will do this by further subdividing into four cases \ref{case2a}--\ref{case2d} based on the value of the conductor $p$ and the length of $w=-\log{u}$.
%
%
%
%

Given that $I$ is principal, we may write $I = f O_F$ for some $f \in F^{\times}$. In this case we have that 
$$D=(I,u)=(f O_F,u) = (f O_F, |f|^{-1})+ (O_F, u |f|)=(f^{-1})+ (O_F, u'),$$
where $(f^{-1})$ is the principal Arakelov divisor generated by $f^{-1}$ and 
$u'= u |f|= (u_i |\sigma_i(f)|)_i \in (\mathbb{R}^{\times}_{+})^3$.
Thus $D$ and $D'=(O_F, u')$ are in the same class of divisors in $\Pic^0_F$. In other words, we have
$h^0(D) = h^0(D')$. Therefore, without loss of generality we can assume that $D$ has the form $(O_F, u)$ for some $u \in (\mathbb{R}^{\times}_{+})^3$ and $N(u)=1$.

With the notations in Section \ref{sec2}, the vector $u$ can be chosen such that  $w =-\log{u} \in \mathcal{F}$. Thus  
$$w= \alpha_1 \cdot b_1 + \alpha_2 \cdot b_2 \text{ for some } \alpha_1, \alpha_2 \in \left(-\frac{1}{2}, \frac{1}{2}\right].$$
Denote by 
$$B'(w) = \{ f \in O_F^{\times}: \|u f\|^2 < 3\cdot 2^{2/3} \},$$
and  $ v_f = \log f \in \Lambda = \log O_F^{\times}$ for each $f \in B'(w)$. 

To prove Theorem \ref{thmmain} for cases \ref{case2a}--\ref{case2c}, it is sufficient to show  that $S_1  < 0.000147634$ for all $w=(x, y, z) \ne (0,0,0)$ (see Lemma \ref{S1}). This can be done by finding a suitable lower bound for $\|u f\|^2$ for each $f \in B'(w)$ by the following lemma.

\begin{lemma}\label{Bomega2}
	Assume that  $w =-\log{u} \in \mathcal{F}$. We have 
	$$S_{1} = \sum_{f \in B'(w)}e^{-\pi \|u f\|^2 }.$$
\end{lemma}
\begin{proof}
Suppose that $f \in O_F \backslash \{0\}$ and $\|u f\|^2 < 3\cdot 2^{2/3}$. One has $N(u f) = N(f)$ since $N(u)=1$.  
It follows that
$$3\cdot 2^{2/3} > \|u f\|^2 \geq 3 |N(u f)|^{2/3} = 3 |N(f)|^{2/3}.$$
This implies that $|N(f)|=1$. In other words, $f \in O_F^{\times}$.	
\end{proof}
By choosing $w \in \mathcal{F}$, one obtains that $\|w\| \le \sqrt{3} \lambda_1/2$. We first consider the case $p\ge 9$, then similarly the case $p=7$.

%% file: case2a.tex
Lemma \ref{lambda1} provides that $\lambda_1 \ge 1.296382$. 
The lower bound on $\|w\|$ leads to $\|u\|^2 \ge 3.194928$. 

Let $ f \in B'(w)$. Then
$\| \log (uf)\| = \| \log f + \log u \|= \|v_f - w\|$.
It follows that  $\|v_f - w\| < \lambda_1$  since otherwise $\| u f\|^2 > 3\cdot 2^{2/3}$. Thus,  $ f \in B(w)$. Therefore $B'(w) \subset B(w)$. 

By Lemma \ref{Bomega}, 
$B'(w)  \subset \{\pm 1, \pm \mathbf{x}_1, \pm\mathbf{x}_2, \pm \mathbf{x}_3\} \subset O_F^{\times}$ where 
$$\|\log \mathbf{x}_1-w\| \ge 0.561350, \|\log\mathbf{x}_2-w\| \ge 0.648191 \text{ and } \|\log \mathbf{x}_3-w\| \ge   1.122700.$$
Since $\log \mathbf{x}_i-w = \log(u \mathbf{x}_i)$ for all $1 \le i \le 3$, we obtain the following.
$$\| u \mathbf{x}_1\|^2 \ge 3.568526, \| u \mathbf{x}_2\|^2 \ge 3.742282 \text{ and } \| u \mathbf{x}_3\|^2 \ge 5.161825>3\cdot 2^{2/3}.$$
As a consequence, $\mathbf{x}_3 \not \in B'(w) $ and then  
$$S_1 \le   \sum_{f \in B(w)\backslash\{\pm \mathbf{x}_3\} }e^{-\pi \|u f\|^2 } \le 2 \left(e^{-\pi \|u \|^2}+ e^{-\pi \|u \mathbf{x}_1 \|^2} + e^{-\pi \|u \mathbf{x}_2\|^2}  \right)$$
$$ \le  2 \left(e^{-3.194928\pi }+ e^{- 3.568526 \pi}  +e^{-3.742282\pi } \right)< 0.00014.$$

%% file: case2b.tex
	The lower bound on $\|w\|$ implies that $\|u\|^2 \ge 3.055940$.
Lemma \ref{lambda1} says that $\lambda_1 \ge 1.296382$. As the proof in \ref{case2a}, $ B'(w) \subset B(w) \backslash\{\pm \mathbf{x}_3\}$. Hence $ \#B'(w) \le \# B(w) -2 \le 6$ by Lemma \ref{Bomega}.

 For each $f \in B'(w)\backslash\{\pm 1\}$, we have
$$\| \log (u f)\| = \| \log f + \log u\|= \|v_f - w\| \ge | \|v_f\| - \|w \|| \ge \lambda_1 -0.324096 \ge 0.972286.$$
 It follows that  $\|u f\|^2 \ge 4.628260$. Consequently, 
$$S_1 = \sum_{f \in B'(w)}e^{-\pi \|u f\|^2 } =2 e^{-\pi \|u \|^2}+ \sum_{f \in B'(w)\backslash\{\pm 1\}} e^{-\pi \|u f\|^2 }$$
$$ \le 2 e^{-3.055940 \pi }+ 6e^{- 4.62826\pi }<0.00014.$$

%% file: case2c.tex
 Lemma \ref{lambda1} shows that $\lambda_1 \ge 1.025134$.   
By an argument similar as the case $ p \ge 9$, one obtains that
$$S_1 \le   \sum_{f \in B(w) }e^{-\pi \|u f\|^2 } \le 2 \left(e^{-\pi \|u \|^2}+ e^{-\pi \|u \mathbf{x}_1 \|^2} + e^{-\pi \|u \mathbf{x}_2\|^2} + e^{-\pi \|u \mathbf{x}_3\|^2} \right),$$
and 
 $\| u \mathbf{x}_3\|^2\ge 4.36359 $.\\
 The lower bounds for $\| u \|^2$, $\| u \mathbf{x}_1\|^2$ and $\| u \mathbf{x}_2\|^2$ and an upper bound for $S_1$ are provided in the following table.
 
 \input{table2c}

 Note that $\lambda_1/6 \ge 0.170856 $. Table \ref{table1} shows that $S_1 \le 0.000146$ for all $0.170856  \le \|w\| \le  \sqrt{3} \lambda_1/2$. Then Theorem \ref{thmmain} is claimed by Lemma \ref{S1}.   

%% file: table2c.tex
\begin{center}
\captionof{table}{$p =7$ and $0.170856  \le \|w\| \le  \sqrt{3} \lambda_1/2$}
    \label{table1}
    \resizebox{\textwidth}{!}{
     \begin{tabular}{|c| c |c|c|c|c |c|}
         \hline
          $\|w\|$ & $\| u \|^2 $ & $ \|\log \mathbf{x}_1-w\|$ & $\| u \mathbf{x}_1\|^2 $ & $ \|\log \mathbf{x}_2-w\|$ & $\| u \mathbf{x}_2\|^2$  & $S_1 $ \\ [0.5ex] 
       
           $\in $ & $ \ge $ & $ \ge $ & $ \ge $ & $ \ge$ & $\ge$  & $ \le$ \\ [0.5ex] 
                           
          \hline 
         $[\lambda_1/2, \sqrt{3} \lambda_1/2]$  & $3.47238$ & $\sqrt{3} \lambda_1/4 $ & $3.35804$ & $\lambda_1/2$ & $3.47238$  & $ 0.000142$ \\ 
         
           &  & $ \approx 0.443896$ &  &   $\approx 0.512567$ &  &  \\ 
           
        \hline 
        $[3\lambda_1/8, \lambda_1/2)$  & $3.27124$ & $\lambda_1/2$ & $3.47238$ & $\lambda_1/2$ & $3.47238$  & $ 0.000145$ \\ 
                 
        &  & $ \approx 0.512567$ &  &   $ \approx 0.512567$ &  &  \\

       \hline 
        $[\lambda_1/4, 3\lambda_1/8)$  & $3.12354$ & $5\lambda_1/8$ & $3.72586$ & $5\lambda_1/8$ & $3.72586$  & $ 0.000145$ \\ 
                   
        &  & $ \approx 0.64070$ &  &   $ \approx 0.64070$ &  &  \\

          
                
              
                
        \hline 
        $[\lambda_1/5, \lambda_1/4)$  & $3.079938$ & $3\lambda_1/4$ & $4.031758$ & $3\lambda_1/4$ & $4.031758$  & $ 0.000141$ \\ 
                      
             &  & $ \approx 0.768851$ &  &   $ \approx 0.768851$ &  &  \\

        \hline 
        $[\lambda_1/6, \lambda_1/5)$  & $3.055939$ & $4\lambda_1/5$ & $4.169016$ & $4\lambda_1/5$ & $4.169016$  & $ 0.000146$ \\ 
                        
       &  & $ \approx 0.820107$ &  &   $ \approx 0.820107$ &  &  \\      
       \hline                         
        \end{tabular}
        }
\end{center}

%% file: case2d1.tex
 We rewrite $u$ as $u= (e^x, e^y, e^z)$ where $w = -\log u =(-x, -y, -z) \in \mathbb{R}^3$ and $x+ y + z = 0$.
 Now let $f \in O_F$ and denote by $f_i = \sigma^i(f)$ for $ i = 0, 1, 2$. Then 
   $$\|u f\|^2= e^{2 x} |\sigma^0(f)|^2 + e^{2 y} |\sigma(f)|^2 + e^{2 z} |\sigma^2(f)|^2 =  f_0^2 e^{2 x} + f_1^2 e^{2 y}  +  f_2^2 e^{2 z}.$$
    We now set
  \begin{equation*}\label{eqG1}
    G_1(u,f) =e^{-\pi[ \|u f\|^2 - \|f\|^2]} - 1= e^{-\pi [ (e^{2 x}-1)f_0^2 +  (e^{2 y}-1)f_1^2  +  (e^{2 z}-1)f_2^2 ]} -1,
  \end{equation*}
  
  \begin{equation*}\label{eqG2}
  G_2(u,f) = G_1(u,f) + G_1(u,\sigma(f)) + G_1(u,\sigma^2(f)) \text{ and }
 \end{equation*}
 
\begin{equation*}\label{eqG}
  G(u,f) = e^{-\pi \|f\|^2} G_2(f,u)/\|w\|^2 \text{ for all } f \in O_F,
 \end{equation*}
 	$$T_1(u)= G(u, 1) + G(u, -1)=2 G(u, 1),$$ 
 	$$T_2(u)=\sum_{f \in O_F, \hspace*{0.1cm}\|f\|^2 \ge 10 } G(u, f) \hspace*{1cm} \text{    and    } \hspace*{1cm} T_3(u)=\sum_{f \in O_F\backslash\{0, \pm 1\}, \hspace*{0.1cm}\|f\|^2 < 10 } G(u, f).$$

 	By Proposition \ref{equiv}, to prove Theorem \ref{thmmain}, it is sufficient to show that for all $ u= (e^x, e^y, e^z) \text{ with } x+y + z = 0 \text{ and } \|w\|^2 = x^2 + y^2 +z^2 \in (0, 0.170856^2)$,
	$$\sum_{f\in O_F} G(u,f) <0.$$
	Since $\sum_{f \in O_F } G(u, f)  = T_1(u) + T_2(u) + T_3(u)$, the last inequality is equivalent to $ T_1(u) + T_2(u) + T_3(u) <0$, which is true by Propositions \ref{Gat1}--\ref{sum3}.

%% file: case2d2.tex
We now complete the proof of our main theorem in this case by proving the following results that can be achieved by using the Galois property of $F$, the Taylor expansion of the function $e^t$ and therefore the symmetry of  $G_2(u,f)$.

\begin{lemma}\label{G}
	We have $G(u,f) = G(u,\sigma(f)) =  G(u,\sigma^2(f))$ for all $u \in (\mathbb{R}_{>0})^3$ and $ f\in O_F$.
\end{lemma}
\begin{proof}
	This can be easily seen by writing down the formulas of $G(u,f)$,  $G(u,\sigma(f))$,  $G(u,\sigma^2(f))$ and by using the fact that $\|f \| = \| \sigma(f)\| = \|\sigma^2(f)\|$ for all $ f\in O_F$.
\end{proof}

\begin{proposition}\label{equiv}
	The conclusion of Theorem \ref{thmmain} is equivalent to the following.
	$$\sum_{f\in O_F} G(u,f) <0 \text{ for all } u= (e^x, e^y, e^z) \ne (1,1,1).$$
\end{proposition}

\begin{proof}
	The assumption $u= (e^x, e^y, e^z) \ne (1,1,1)$ implies that $\|w\|^2 = x^2 + y^2 + z^2 > 0$.  
	We have
	\begin{multline}\label{eqh1}
		k^0(D) - k^0(D_0) = \sum_{f \in O_F } \left( e^{-\pi \|u f\|^2} - e^{-\pi \|f\|^2}\right) = \sum_{f \in O_F } e^{-\pi \|f\|^2}\left( e^{-\pi[ \|u f\|^2 - \|f\|^2]} - 1\right)\\	
		=\sum_{f \in O_F } e^{-\pi \|f\|^2} G_1(u, f) \text{ for all } u \in (\mathbb{R}_{>0})^3.
	\end{multline}

	Since $F$ is Galois, we have $\sigma(O_F) = O_F$ and then similar results are obtained as below.
	\begin{multline}\label{eqh2}
		k^0(\sigma(D)) - k^0(\sigma(D_0)) = \sum_{\sigma(f) \in O_F } \left( e^{-\pi \|u \sigma(f)\|^2} - e^{-\pi \|\sigma(f)\|^2}\right) \\
		=\sum_{f \in O_F } e^{-\pi \|\sigma(f)\|^2} G_1(u,\sigma(f)) 
	\end{multline}
	
	\begin{multline}\label{eqh3}
		k^0(\sigma^2(D)) - k^0(\sigma^2(D_0)) = \sum_{\sigma^2(f) \in O_F } \left( e^{-\pi \|u \sigma^2(f)\|^2} - e^{-\pi \|\sigma^2(f)\|^2}\right) \\
		=\sum_{f \in O_F } e^{-\pi \|\sigma^2(f)\|^2} G_1(u,\sigma^2(f)).
	\end{multline}

	Lemma \ref{h0sym} says that $k^0(D) = k^0(\sigma(D)) = k^0(\sigma^2(D))$ and $k^0(D_0) = k^0(\sigma(D_0)) = k^0(\sigma^2(D_0))$. Moreover $\|f \| = \| \sigma(f)\| = \|\sigma^2(f)\|$ by Lemma \ref{lengf}. Taking the sum of \eqref{eqh1}--\eqref{eqh3} then using these equalities, the following result is implied. 
	\begin{multline*}\label{eqh4}
		3[k^0(D) - k^0(D_0) ]=  \sum_{f \in O_F } e^{-\pi \|f\|^2} G_2(u, f) 
		= \sum_{f \in O_F } G(u, f) \|w\|^2\text{ for all } u \in (\mathbb{R}_{>0})^3.
	\end{multline*}
	Hence, the following equivalences hold for all $u= (e^x, e^y, e^z) \ne (1,1,1).$
	$$k^0(D) < k^0(D_0) \Leftrightarrow [k^0(D)-k^0(D_0)]/\|w\|^2<0  \Leftrightarrow \sum_{f \in O_F } G(u, f)<0.$$
\end{proof}




\begin{proposition}\label{taylor1}
	Let $ u= (e^x, e^y, e^z) \text{ with } x+y + z = 0 \text{ and with } \|w\|^2 = x^2 + y^2 +z^2>0$. Then for all  $f \in O_F$,
	$$G(u, f) \le 4 \pi^2 \|f^2\|^2 e^{-\pi \|f\|^2} \left(1 + \frac{1}{2} e^{2 \pi \|w\| \|f ^2\|} \right).$$
	In particular,  if $f \in O_F$ with $\|f\|^2 \ge 9$ then
	$$G(u, f) \le 4 \pi^2\left(e^{-(\pi-1/2)\|f\|^2}+ \frac{1}{2} e^{-\pi(1-2\|w\|-1/(2 \pi)) \|f\|^2} \right).$$
\end{proposition}

\begin{proof}
	Since $ e^t -1 \ge t $ for all $t \in \mathbb{R}$, the following holds for all $x, y , z \in \mathbb{R}$.
	$$G_1(u,f) \le e^{-2 \pi ( x f_0^2 +  y f_1^2  + z f_2^2 )} -1.$$
	The Taylor expansion of $e^{-2 \pi ( x f_0^2 +  y f_1^2  + z f_2^2 )} -1$ provides that
	$$G_1(u,f) \le -2 \pi ( x f_0^2 +  y f_1^2  + z f_2^2 )  + \sum_{k \ge 2} \frac{1}{k!} \left[-2 \pi ( x f_0^2 +  y f_1^2  + z f_2^2 )\right]^k.$$
	Each term in the later sum can be bounded as below. 
	$$|x f_0^2 +  y f_1^2  + z f_2^2 | \le \sqrt{x^2+y^2+z^2} \sqrt{f_0^4+ f_1^4+ f_2^4} = \|w\| \|f^2\|.$$ 
	Thus, 
	$$G_1(u,f)\le -2 \pi ( x f_0^2 +  y f_1^2  + z f_2^2 )  +  \sum_{k \ge 2} \frac{1}{k!} \left( 2 \pi \|w\| \|f^2\| \right)^k.$$
	Since we can write $1/2 = 1/3 + 1/6$ and since $1/k! \le (1/6) [1/(k-2)!]$ for any $k \ge 3$, the last sum in $G_1(u,f)$ is less than or equal to
	$$\left( 2 \pi \|w\| \|f^2\| \right)^2 
	\left( \frac{1}{3} + \frac{1}{6} \sum_{k \ge 0} \frac{1}{k!} \left( 2 \pi \|w\| \|f^2\| \right)^{k}\right)=4 \pi^2 \|w\|^2 \|f^2\|^2 \left( \frac{1}{3} + \frac{1}{6} e^{ 2 \pi \|w\| \|f^2\| } \right).$$
	
	Therefore
	\begin{equation}\label{eq2c1}
		G_1(u,f)\le -2 \pi ( x f_0^2 +  y f_1^2  + z f_2^2 )  + \frac{4}{3} \pi^2 \|w\|^2 \|f^2\|^2
		\left( 1 + \frac{1}{2} e^{ 2 \pi \|w\| \|f^2\| } \right).
	\end{equation}

	Similarly, we obtain upper bounds for $G_1(u,\sigma(f)) $ and $G_1(u,\sigma^2(f)) $ as follows.
	\begin{equation}\label{eq2c2}
		G_1(u,\sigma(f)) \le -2 \pi ( x f_1^2 +  y f_2^2  + z f_0^2 )  + \frac{4}{3} \pi^2 \|w\|^2 \|f^2\|^2
		\left( 1 + \frac{1}{2} e^{ 2 \pi \|w\| \|f^2\| } \right)
	\end{equation}
	and 
	\begin{equation}\label{eq2c3}
		G_1(u,\sigma^2(f)) \le -2 \pi ( x f_2^2 +  y f_0^2  + z f_1^2 )  + \frac{4}{3} \pi^2 \|w\|^2 \|f^2\|^2
		\left( 1 + \frac{1}{2} e^{ 2 \pi \|w\| \|f^2\| } \right).
	\end{equation}
	In these bounds, we again use Lemma \ref{lengf} to replace $\|\sigma^2(f^2)\|$ and $\|\sigma(f^2)\|$ with $\|f^2\|$.  
	Taking the sum of the right hand side parts of \eqref{eq2c1}--\eqref{eq2c3} and using the condition that $x +y +z =0$, the following is implied. 
	$$G_2(u,f) = G_1(u,f) + G_1(u,\sigma(f)) + G_1(u,\sigma^2(f)) 
	\le 4 \pi^2 \|w\|^2 \|f^2\|^2 \left( 1 + \frac{1}{2} e^{ 2 \pi \|w\| \|f^2\| } \right).$$
	The first part of the proposition then follows since $G(f,u) = e^{-\pi \|f\|^2} G_2(u,f)/\|w\|^2$.
	The second part is obtained by using the fact that
	$$\|f^2\|^2 \le \|f\|^4 \le e^{\|f\|^2/2}  \text{ and } \|f^2\| \le \|f\|^2 \text{ for all } \|f\|^2 \ge 9.$$
\end{proof}

\begin{proposition}\label{Gat1}
Let $ u= (e^x, e^y, e^z) \text{ with } x+y + z = 0 \text{ and with } \|w\|^2 = x^2 + y^2 +z^2 \in (0, 0.170856^2)$. Then $T_1(u) \le -0.002652393$.
\end{proposition}

\begin{proof}
It is true for any $ 0< \|w\|<0.170856 $ that
$$e^{2 x} + e^{2 y} +  e^{2 z}-3 \ge 1.9 (x^2 + y^2 + z^2)=1.9 \|w\|^2.$$
Consequently,
$$G_1(u,1)= e^{-\pi [ e^{2 x} + e^{2 y} +  e^{2 z}-3]} -1 \le e^{- 1.9 \pi \|w\|^2} -1.$$ 
Thus
$$G(u,1)=G_2(u,1)/\|w\|^2 = 3 G_1(u,1)/\|w\|^2 \le 3[e^{- 1.9 \pi \|w\|^2} -1]/\|w\|^2.$$
Since $0< \|w\|<0.170856$, we obtain that $G(u,1) \le -0.001326196$.
Therefore $T_1(u) = 2 G(u,1) \le -0.002652393$.
\end{proof}

\begin{proposition}\label{sum2}
Let $ u= (e^x, e^y, e^z) \text{ with } x+y + z = 0 \text{ and with } \|w\|^2 = x^2 + y^2 +z^2 \in (0, 0.170856^2)$. Then $T_2(u)<0.000461879$.
\end{proposition}

\begin{proof}
By Proposition \ref{taylor1}, 
$$T_2(u) \le 4 \pi^2 \sum_{f \in O_F, \|f\|^2 \ge 10} 
\left(e^{-(\pi-1/2)\|f\|^2}+ \frac{1}{2} e^{-\pi(1-2\|w\|-1/(2 \pi)) \|f\|^2} \right).$$
The first sum is at most $0.001 \cdot 10^{-6}$ by Corollary \ref{sum}. 
Moreover,
$$\pi(1-2\|w\|+1/(2 \pi)) \ge 1.568074 \text{ since } \|w\|<0.170856.$$
 Hence, the second sum is bounded by
 $$\sum_{f \in O_F, \|f\|^2 \ge 10 } e^{- 1.568075\pi \| f\|^2}$$
which is at most $23.399 \cdot 10^{-6}$ (see Corollary \ref{sum}). Thus $T_2(u) \le 0.000461879$.

\end{proof}

\begin{proposition}\label{sum3}
Let $ u= (e^x, e^y, e^z) \text{ with } x+y + z = 0 \text{ and with } \|w\|^2 = x^2 + y^2 +z^2 \in (0, 0.170856^2)$. Then  $T_3(u)<0.00138339.$
\end{proposition}

\begin{proof}
In case $p \ge 19$, Proposition \ref{minlength} says that $\|f\|^2 \ge 13$ for all $f \in O_F\backslash \mathbb{Z}$. Therefore $T_3(u) =0$.

Now we consider the case in which $ p \le 13$. It is easy to find all vectors $f \in O_F$ for which $\|f\|^2 < 10$ using
 an LLL-reduced basis of the lattice $O_F$ (see Section 12 in \cite{ref:1}) or by applying the Fincke--Pohst algorithm (see Algorithm 2.12 in \cite{ref:40}). 
 
If $p=13$ then there are 6 vectors $f \in O_F$ for which $\|f\|^2 < 10$. They have the forms $ \pm g, \pm \sigma(g), \pm \sigma^2(g)$ with $\|g\|^2= 9$ and $\|g^2\|^2= 53$. Applying Proposition \ref{taylor1} leads to
$$T_3(u)= 6G(u, g) \le 24 \pi^2 \|g^2\|^2 e^{-\pi \|g\|^2} \left(1 + \frac{1}{2} e^{2 \pi \|w\| \|g ^2\|} \right)$$
$$= 24 \pi^2( 53)  e^{- 9 \pi } \left(1 + \frac{1}{2} e^{2 \pi (0.170856) \sqrt{53} } \right) <0.00138339.$$

Similarly, one can show that $T_3(u) < 0.00138339$  in case  $p=9$.

Finally, if $p=7$ then $F$ is the splitting field of the polynomial $X^3-X^2-2X+1$. Let $\theta$ be a root of this polynomial. There are 12 vectors $f \in O_F$ for which $\|f\|^2 < 10$. Those are  
$$ \pm \theta, \pm \sigma(\theta), \pm \sigma^2(\theta),  \pm (1+\theta), \pm \sigma(1+\theta), \pm \sigma^2(1+\theta)$$
 with 
 $$\|\theta\|^2=5,  \|1+\theta\|^2=6;  \|\theta^2\|^2= 13, \|(1+\theta)^2\|^2= 26.$$
We have $T_3(u)= 6[G(u, \theta) + G(u, 1+\theta)]$. 
Substitute the coordinates of $\theta$ to the formulas of $G(u,\theta)$  and $G(u,1+\theta)$  in \eqref{eqG} and find the maximum of $T_3(u)$ with the conditions in the proposition, we obtain  that $T_3(u)< 0.00138339$.
\end{proof}

\begin{remark}
In the proof of Proposition \ref{sum3}, we do not use Proposition \ref{taylor1} when $p=7$. The reason for this is that the upper bound it provides for $G(u, \theta) $ and $G(u, 1+\theta)$ and hence for $T_3(u)$ are too large to show that $T_1(u) + T_2(u) + T_3(u) <0$ (used in the proof of Theorem \ref{thmmain} in case \ref{case2d}).
\end{remark}

%% file: compare.tex
Here we give a summary of the similarities and the differences between this work and the previous work  \cite{ref:14,ref:15} and \cite{Tran2}.

The overall structure of our proof is similar to that of previous work. In particular, we consider separately the cases where $I$ is principal and where it is not, and in the later case, we subdivide further based on the relative length of $w$.
Both in our work and in previous work, the case where $I$ is not principal is handled by using the fact that the squared length of any vector in the lattice associated to $D$ is at leat $n \cdot 2 ^{n/2}$ where $n$ is the degree of the number field (see Section \ref{case1}).
The proofs are also structurally similar in the case where $I$ is principal and $w$ is not too short, in this case we use bounds on the size of the fundamental unit as well as bounds on the number of short vectors of the lattice associated to $D$ (see Remark \ref{lambdahexan}, Proposition \ref{sum3} and  Corollary \ref{sum}).

The major difference between our proof and those appearing in previous work arises because in the present case the unit group has rank two whereas in previous cases it was one.
In the case $I$ is principal and $w$ is short the strategy of the prior work was to apply the standard theory of optimizing single variable differentiable functions, that is to check derivative conditions.
In our case, we must do more work. Indeed, when $I$ is principal and $w$ is short, we proved directly that  $h^0(D)< h^0(D_0)$. In order to do this, we had to make use of explicit information about the structure of $O_F$ to get a lower bound on the lengths of vectors in $O_F\backslash \mathbb{Z}$ based on the conductor of $F$. In addition, the Galois-invariance of $h^0(D)$, the symmetry of  $G_2(u,f)$ and the Taylor expansion of the function $e^t$ (see Propositions \ref{equiv} and \ref{taylor1}) are all employed in the proof.
In the case $I$ is principal and $w$ is not too short, we had to exploit again explicit information about the structure of $\Lambda$, namely that it is a hexagonal lattice, to obtain an upper bound on $S_1$ (see Lemma \ref{Bomega}), in previous work, because this lattice had rank one, this entire question was figured out easier. 

%% file: furtherwork.tex
It is natural to question whether our method can be applied to other number fields $F$ which satisfy the hypothesis of the conjecture mentioned in Section \ref{sec1a}. Indeed, with the notations in the earlier sections, it still works in the case in which $I$ is not principal (see Section \ref{case1} and in \cite{ref:14,ref:15}, \cite{Tran2}) by Proposition 4.4 in \cite{ref:21}. 
In addition, when $I$ is principal and $w$ is short, involving few cumbersome estimations and modifications according to the degree $n$ of $F$, one can also prove this conjecture using the same method presented in Section \ref{case2d}. 

However, our method may fall short in being applied in other cases. That is because it requires a good knowledge of the structure of the unit lattice $\Lambda$ such as its $\mathbb{Z}$-basis, the length of its shortest vectors as well the points in $\Lambda$ close to a given point in its fundamental domain (see Lemma \ref{Bomega}), together with an efficient bound on the number of vectors of length bounded in the lattice $O_F$. Since these are not always known for $F$, a further research addressing a new method may be needed.